\documentclass[sn-mathphys]{sn-jnl}

\usepackage{xspace}
\jyear{2021}%

\theoremstyle{thmstyleone}%
\newtheorem{theo}{Theorem}
\newtheorem{lemma}[theo]{Lemma}

\newtheorem{proposition}[theo]{Proposition}%

\theoremstyle{thmstyletwo}%
\newtheorem{remark}[theo]{Remark}

\theoremstyle{thmstylethree}%

\renewcommand{\epsilon}{\varepsilon}
\newcommand{\X}{\mathcal{X}}
\renewcommand{\w}{m}
\newcommand{\n}{n}
\newcommand{\G}{g}
\newcommand{\g}{\phi}
\newcommand{\N}{\mathbb{N}}
\renewcommand{\P}{\mathbb{P}}
\newcommand{\E}{\mathbb{E}}

\newcommand{\V}{\mathbb{V}}
\newcommand{\Ct}{M_t}

\newcommand{\St}{S_t}

\newcommand{\hXt}{\hat X_t}
\newcommand{\wCt}{\tilde M_t}

\newcommand{\mm}{MMM\xspace}
\newcommand{\sm}{FMM\xspace}
\begin{document}

\title[Branching with selection and mutation I]
{Branching with selection and mutation I:\\ Mutant fitness of Fr\'echet~type}


\author[1]{\fnm{Su-Chan} \sur{Park}}\email{spark0@catholic.ac.kr}

\author[2]{\fnm{Joachim} \sur{Krug}}\email{jkrug@uni-koeln.de}

\author[3]{\fnm{L\'eo} \sur{Touzo}}\email{leo.touzo@phys.ens.fr}
\author*[4]{\fnm{Peter} \sur{M\"orters}}\email{p.moerters@uni-koeln.de}

\affil[1]{\orgdiv{Department of Physics}, \orgname{The Catholic University of Korea}, \orgaddress{\street{43 Jibong-ro}, \city{Bucheon-si}, \postcode{14662}, \state{Gyeonggi-do}, \country{Republic of Korea}}}

\affil[2]{\orgdiv{Institut f\"ur Biologische Physik}, \orgname{Universit\"at zu K\"oln}, \orgaddress{\street{Z\"ulpicher Str. 77}, \city{50937 K\"oln}, \country{Germany}}}

\affil[3]{\orgdiv{Laboratoire de Physique de l'\'Ecole Normale Sup\'erieure}, \orgname{ENS, Université PSL, CNRS, Sorbonne Universit\'e, Universit\'e Paris Cité}, \orgaddress{\street{}\postcode{F-75005} \city{Paris}, \state{}\country{France}}}

\affil*[4]{\orgdiv{Mathematisches Institut}, \orgname{Universit\"at zu K\"oln}, \orgaddress{\street{Weyertal 86--90}, \city{50931 K\"oln},   \country{Germany}}}


\abstract{We investigate two stochastic models of a growing population subject to selection and mutation.
In our models each individual carries a fitness which determines
its mean offspring number. Many of these offspring inherit their parent's  fitness, but some are mutants and 
obtain a fitness randomly sampled from a distribution in the domain of attraction of the Fr\'echet distribution. 
We give a rigorous proof for the precise
rate of superexponential growth of these stochastic processes and support the argument by a heuristic and numerical study of the mechanism underlying this growth. }




\maketitle
\section{Introduction}

While the theory of branching processes is undoubtedly one of the best developed areas of probability theory, stochastic branching models
that incorporate effects of selection and mutation have only recently caught the attention of mathematicians and physicists. This is despite the
unquestionable relevance of these effects to the evolution of populations in nature and in the laboratory~\cite{Park2010,Wiser2013}. 
\medskip

By contrast, deterministic high density models of a population undergoing selection and mutation have been studied for quite some time. The model most closely associated with our stochastic process is Kingman's model~\cite{Kingman1978}. This is a dynamical system on the space of probability measures describing the fitness distribution of a population. After $t$ generations the fitness distribution $p_t$ of the population is replaced by
$$p_{t+1}(dx) = (1-\beta) \frac{ x \, p_{t}(dx)}{\int x \, p_{t}(dx)} + \beta \mu(dx).$$
Here a proportion $1-\beta$ of the new generation has been selected from the current generation proportionally to their fitness
and a proportion $\beta$ are mutants that get a new fitness, sampled independently from their past using the mutant fitness 
distribution $\mu$. Note that this model is only well-defined if the mean fitness remains finite and therefore requires moment bounds for
the mutant fitness distribution. 
Kingman's model undergoes a condensation phase transition,
which is further studied in~\cite{Dereich2013}. Variants of the model 
have been considered in~\cite{Yuan,Yuan2020,Yuan2022}. 
\medskip

The study of individual based, stochastic models is much more recent. Park and Krug~\cite{Park2008} studied Kingman's model for an unbounded 
fitness distribution alongside a random model of fixed, finite population size. 
Despite its highly simplified nature, the finite population model with an exponential distribution
of fitness effects is qualitatively consistent with the fitness increase observed in Lenski's long-term
evolution experiment with bacteria \cite{Sibani1998,Wiser2013}.
A generalization of this model that includes the response of the immune system to a population
of pathogens was considered in~\cite{Bianconi2011}.
\medskip

The first papers studying branching process models
that express the selective advantage of a fit individual in terms of its offspring distribution are~\cite{Dereich2017}, which deals
with Weibull type fitness distributions and puts the focus on the condensation phenomenon in that model, and~\cite{Mailler2021}
which looks at the growth of the fittest family in the case of Gumbel type distributions. Both papers are limited to bounded 
fitness distributions and implicitly rely on the analogy to Kingman's original model, though of course the methods of study 
are entirely different in a stochastic setting. The present paper initiates the study of branching processes with selection and mutation for
unbounded fitness distributions. We focus on the case of Fr\'echet type fitness distributions where the mathematical challenge is linked to 
the fact that the analogous Kingman model is ill-defined \cite{Park2008}.
\medskip

The structure of the paper is as follows.
In Sec.~\ref{Sec:model}, we introduce the models and state the main theorem. Sec.~\ref{Sec:sim} explains the heuristics behind
the formal results and in Sec.~\ref{Sec:pf} we present a rigorous proof of the theorem. Sec.~\ref{Sec:steady} contains refined 
results for the empirical frequency distribution for one of our models. These results are not yet accessible by a complete rigorous mathematical 
analysis, so that we resort to a numerical and heuristic study and a rigorous analysis of an approximating deterministic system. In Sec.~\ref{Sec:sum} we provide a short discussion that places our results into the context
of previous work and points to directions for future research.
\pagebreak[3]

\section{\label{Sec:model}Models and main result} 
We study two models of a population evolving in discrete generations. In both models all individuals are assigned a fitness value, which is 
a positive real number. As model parameters we fix a probability distribution $\mu$ on $(0,\infty)$ from which the random fitness values $F$ are sampled, 
and a mutation probability~$\beta\in(0,1)$.%
\smallskip

In both models we start from generation $t=0$ with a 
single individual\footnote{The generalization to multiple individuals is straightforward.}
with fitness $f$.
Each individual in the population in generation $t\geq 0$ produces a Poisson random number of 
offspring with mean given by its fitness. With probability $1-\beta$ an offspring individual 
inherits its parent's fitness and is added to the population at generation $t+1$.
Otherwise, with probability $\beta$, it is a mutant. The two models differ in the fate of the mutants.
\begin{itemize}
\item \emph{Fittest mutant model (\sm):} Every mutant is assigned a fitness sampled independently from~$\mu$. 
Only the fittest mutant (if there is one) is added to the population at generation $t+1$. All other mutants die instantly.

\item \emph{Multiple mutant model (\mm):} Every mutant is assigned a fitness sampled independently from~$\mu$ and is added to the population at generation $t+1$.
\end{itemize}

We write $X(t)$ for the number of individuals in generation $t$ and study the growth of the populations conditioned
on the event of survival, i.e. when $X(t) \neq 0$ for all times~$t$. 
It is easy to see that the population size of the
\mm dominates the population size of the \sm at all times. Because the growth is determined by the fittest 
mutants we expect both models to grow at the same rate and to show this, it suffices to find an upper bound for the \mm
and a matching lower bound for the \sm.
\medskip

Naturally, the rate of growth depends on the  mutant fitness distribution~$\mu$. If $\mu$ is an unbounded distribution 
in both models individuals of ever increasing fitness occur and hence the  population will grow 
superexponentially fast. By contrast, if $\mu$ is bounded we can only have exponential growth.  Indeed, if 
$\mu$ is continuous with essential supremum one,
then for a closely related continuous time model of immortal individuals, 
it is shown in~\cite[Remark~1]{Dereich2017}  that
$$\lim_{t\to\infty} \frac{\log X(t)}{t}= \lambda^*,$$
where $\lambda^*\in[1-\beta,1)$ is the unique solution of the equation 
$$\beta \int \frac{\lambda^*}{\lambda^*-(1-\beta)x} \, \mu(dx) = 1$$
if $\beta \int \frac{1}{1-x} \, \mu(dx)\geq 1$, and otherwise $\lambda^*:=1-\beta$. 
Further details on the long term growth of the process depend on the classification
of $\mu$ according to its membership in the max domain of attraction of an
extremal distribution. By the celebrated Fisher-Tippett theorem
there are three such universality classes, see for example~\cite[Proposition~0.3]{Resnick}. These are 
\begin{itemize}
    \item the \emph{Weibull class}, which roughly occurs if $\mu$ is bounded with mass decaying slowly 
    near the essential supremum, 
    \item the \emph{Gumbel class}, which roughly occurs if the mass of $\mu$ is decaying quickly near the essential supremum, 
    which may be finite or infinite,
    \item the \emph{Fr\'echet  class}, which roughly occurs if $\mu$ is unbounded with mass decaying slowly 
    near infinity.
\end{itemize}
The assignment of mutant fitness distributions to extreme value universality classes plays an important role in the interpretation of evolution experiments~\cite{Joyce2008}, and
representatives of all three classes have been identified empirically~\cite{Das2022}.
\medskip

In the present paper, we are mainly interested in the asymptotic behaviour of the population size $X(t)$
in the so far unexplored case that $\mu$ belongs to the 
Fr\'echet class (or, in short, is of \emph{Fr\'echet type}). 
Precisely, this means that the tail function
$$G(x):=\mu((x,\infty))=\mathbb P(F> x)$$
is regularly varying with index $-\alpha$ for some $\alpha>0$. In other words, there exists a function 
$\ell\colon (0,\infty) \to \mathbb R$ which is slowly varying  at infinity such that $G(x)=x^{-\alpha} \ell(x)$.
As in this case $\mu$ is an unbounded distribution, the process $(X(t) \colon t\geq 0)$  will grow 
superexponentially fast on survival and therefore our discussion will focus on the 
limiting quantity 
\begin{align}
\nu = \lim_{t\rightarrow \infty} \frac{\log \log X(t)}{t}.
\label{Eq:defnu}
\end{align}
Our main result is stated in the following theorem.
\begin{theo}\label{Th:mainthm}
Given $\alpha>0$, let $T\in\mathbb N$ be the unique number such that
$$\frac{(T-1)^T}{T^{T-1}} <\alpha \leq \frac{T^{T+1}}{(T+1)^{T}}$$
and define
\begin{align}
\label{Eq:finala}
\nu(\alpha) := \frac1T \log \frac{T}\alpha.
\end{align}
Let $(X(t))_{t\ge0}$ be the size of the population in either the \sm or the \mm.
Then, almost surely on survival,
$$\lim_{t\to\infty} \frac{\log\log X(t)}{t}=\nu(\alpha).$$
\end{theo}
Before presenting the proof of the theorem in Sec.~\ref{Sec:pf}, in the next section we motivate the expression (\ref{Eq:finala}).  

\section{\label{Sec:sim}Motivation of the main result}
Here we explain the statement of Theorem~\ref{Th:mainthm}
by a heuristic analysis of the \sm. 
For convenience we take the fitness distribution to be of Pareto form, $G(x) = x^{-\alpha}$ for $x \ge 1$ and $G(x) = 1$ for $x< 1$. Moreover, throughout this section we assume 
that the initial fitness $f$ is so large that the fluctuations induced by Poisson sampling are negligible at all times, which implies that both the total population size and the
sizes of subpopulations of mutants are well approximated by their expectations.
Denoting the fitness of the mutant that is added to the population at generation $t$
by $W_t$, 
we can then write 
\begin{align}
X(t) \approx (1-\beta)^t f^t + \sum_{i=1}^t (1-\beta)^{t-i} W_i^{t-i},
\label{Eq:evol}
\end{align}
where the factors $1-\beta$ account for the fact that (apart from the added mutant) only the unmutated fraction of the population survives to the next generation.
For the same reason the total number $N_t$ of mutants produced in generation $t$ (including the ones that die immediately) is approximately 
$$
N_t \approx \frac{\beta}{1-\beta} X(t).
$$
Since the probability that the largest fitness $W_t$ among $N_t$ independent and identically distributed 
random variables with common distribution $G$ is smaller than $x$  is $(1-x^{-\alpha})^{N_t}$,
the random variable $W_t$ can be sampled as
$$
W_t = \left ( 1 - Z_t^{1/N_t} \right )^{-1/\alpha}
\approx X(t)^{1/\alpha} Y_t,
\quad Y_t :=  \left (\frac{1-\beta}{\beta} \log \frac1{Z_t} \right )^{-1/\alpha},
$$
where $Z_t$ is uniformly distributed in the interval $(0,1)$
and we have approximated \smash{$Z_t^{1/N_t} \approx 1 + (1/N_t)\log Z_t$}.
Note that $Y_t$ does not depend on $X(t)$.\medskip

To proceed, 
we define $\omega_t$ as
\begin{align*}
\omega_t := \frac{\log X(t)}{\log f},
\end{align*}
which implies that $X(t) = f^{\omega_t}$ and $W_t \approx Y_t f^{\omega_t/\alpha}$. 
Inserting these relations into (\ref{Eq:evol}) we obtain
\begin{align}
f^{\omega_t} \approx (1-\beta)^t f^t + \sum_{i=1}^t (1-\beta)^{t-i} Y_{i}^{t-i}
f^{(t-i)\omega_{i}/\alpha}.
\end{align}
In the limit $f \to \infty$ the sum on the right hand side is dominated by the term with the largest exponent. Correspondingly, the $\omega_t$ can be well approximated by the solution $\chi_t$ of the recursion relation
\begin{align}
\chi_t = \max\left \{t,\frac{t-1}{\alpha}\chi_1, \frac{t-2}{\alpha}\chi_2,\ldots,\frac{t-k}{\alpha}\chi_k,
\ldots, \frac{1}{\alpha}\chi_{t-1}\right \}
\label{Eq:chire}
\end{align}
with $\chi_1=1$. 
We now argue that the $\chi_t$ grow at least exponentially.
Since for any $t_0 \ge 1$ and any positive integer $m$
$$
\chi_{t_0+m} \ge \frac{m}{\alpha} \chi_{t_0},
$$
we have, for any $n \geq 1$
$$
\chi_{t_0+nm} \ge \left ( \frac{m}{\alpha} \right )^n \chi_{t_0}.
$$
Correspondingly
\begin{align}
\lim_{t \rightarrow \infty} \frac{\log \chi_{t}}{t} 
=\lim_{n \rightarrow \infty} \frac{\log \chi_{t_0+nm}}{nm} 
\ge \frac{1}{m} \log \frac{m}{\alpha},
\label{Eq:numm}
\end{align}
where we have assumed that the limit is well-defined.
Since (\ref{Eq:numm}) is valid for any integer $m \ge 1$, an optimal
lower-bound can be found by maximizing the right hand side. As shown by Lemma~\ref{lem:nu} in Sec.~\ref{Sec:pf}, the maximizer over the positive integers is precisely the function $\nu(\alpha)$ in Theorem~\ref{Th:mainthm}. As the population
size depends exponentially on $\omega_t$ or $\chi_t$, the heuristic argument makes it plausible that $\nu(\alpha)$ is a lower bound
on the double-exponential growth rate of $X(t)$. Remarkably, Theorem~\ref{Th:mainthm}
states that the bound is tight, and moreover applies also to the \mm. Informally this implies that the population at time $t$ is dominated by the fittest mutant that was generated at time $t-T$. As a consequence the mutant frequency distribution changes periodically with period $T$ (see Sec.~\ref{Sec:steady} for further discussion). \medskip

In Fig.~\ref{Fig:makea}, we depict $\nu(\alpha)$ together with the numerical
solution\footnote{The direct numerical estimate of $\nu$ by extrapolating $\log \chi_t/t$ is hampered by the fact that
$\lim_{t\rightarrow\infty}\chi_t e^{-\nu t}$ does not exist in general (see Sec.~\ref{Sec:steady}). The method used to obtain the data in Fig.~\ref{Fig:makea}
is explained at the end of Sec.~\ref{Sec:numerical_recursion}.} of the recursion relation~\eqref{Eq:chire}. 
The fact that $\nu(\alpha)$ is the exact exponential growth rate of $\chi_t$ is proven rigorously in Lemma~\ref{ideal_conv} in 
Sec.~\ref{Sec:pf}. In the inset of Fig.~\ref{Fig:makea}, we compare \eqref{Eq:finala}
to an approximation obtained by treating $m$ in \eqref{Eq:numm} as a continuous variable. This yields  
\begin{align}
\max\left \{\frac{1}{x} \log \frac{x}{\alpha} \colon x \ge 1 \right \}
= \begin{cases}
1/(e \alpha), & \alpha e \ge 1,\\
-\log \alpha , & \alpha e < 1.
\end{cases}
\label{Eq:naive}
\end{align}
Although \eqref{Eq:naive} is not exact, the relative error is less than 7\% in all cases. 
\medskip

For $\alpha e < 1$ the expressions (\ref{Eq:finala}) and (\ref{Eq:naive}) actually coincide. In this regime of extremely heavy-tailed fitness distributions (more precisely, in the case of $\alpha \le 0.5$) selection becomes irrelevant, in the sense that the double-exponential growth rate $\nu(\alpha) = \log(1/\alpha)$ persists in the extreme case \mbox{$\beta \to 1$} of the \mm,
where all individuals are replaced by mutants in each generation and the process becomes a classical Galton-Watson process albeit with infinite mean, cf.~\cite{Davies}. 
In the case of the \sm, the extreme case would stop the population from growing but the fitness $W_t$ of the single individual
present approximately satisfies $W_{t+1} \approx W_t^{1/\alpha}$, which gives
$$
\lim_{t\rightarrow \infty} \frac{\log \log W_{t}}{t} = \log(1/\alpha).
$$
\begin{figure}
\centering
\includegraphics[width=\linewidth]{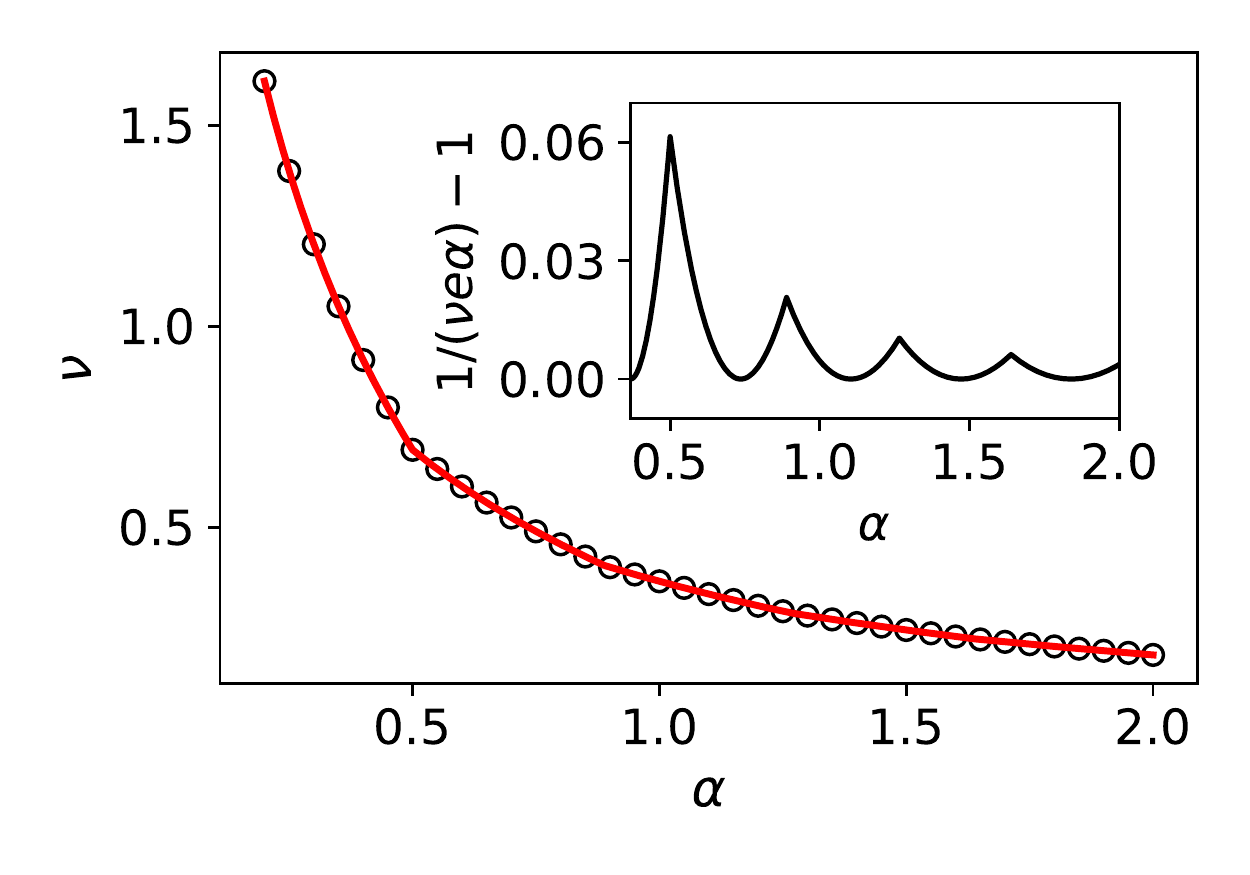}
\caption{\label{Fig:makea} Plots of $\nu $ vs $\alpha$.
Solid line depicts \eqref{Eq:finala} and symbols are from 
numerical solution
of the recursion relation~\eqref{Eq:chire}. Inset: Plot of 
$(\nu e \alpha)^{-1}-1$ vs. $\alpha$ with $\nu $ in \eqref{Eq:finala}.
The error of the approximation \eqref{Eq:naive} is small and
vanishes when $\alpha \leq 1/e$.}
\end{figure}


\section{\label{Sec:pf}Proof of Theorem~\ref{Th:mainthm}}
\subsection{Preparation for the proof}

In this subsection we collect some tools that will be used in the proofs of the lower and upper bounds
in the estimate leading to Theorem~\ref{Th:mainthm}. The lower bound will be verified in Section~\ref{low}
and the upper bound in Section~\ref{up}.
\medskip

For $t\in\N_0$ let $W_t$  be the fitness of the fittest of the mutants in generation~$t$
and $W_t=0$ if there are no mutants in generation~$t$. Our first observation is that
under the weak assumption $G(x) >0$ for all large~$x$ (which always holds if $\mu$ 
is of Fr\'echet type) either the sequence $(W_t)$ is unbounded or the branching 
process dies out in finite time. Heuristically speaking, on survival the accumulated number of mutants is unbounded almost surely, which naturally entails unbounded largest fitness.\\[-0.7cm]

\begin{lemma}
Almost surely on survival the sequence $(W_t)$ is unbounded.\label{Th:Wun}
\end{lemma}
\ \\[-1.2cm]

\begin{proof}
We first show that the branching process can be coupled to a sequence $(\xi_1, \ldots, \xi_t)$
of independent Bernoulli variables with success parameter $\beta$ and an independent 
sequence $(F_1,\ldots, F_t)$ of independent fitnesses with distribution~$\mu$
such that \emph{on survival up 
to generation $t$} we have, for all $1\leq i \leq t$,
\begin{itemize}
\item $\xi_i=1$ if there is at least one mutant in generation $i$, and 
\item $W_i \geq F_i \xi_i$.
\end{itemize}
Indeed, once the random variables $(F_1,\ldots, F_t)$  and $(\xi_1, \ldots, \xi_t)$ are generated with the given law we generate the branching process as follows: Produce the offspring in the $n$th generation 
as a Poisson distribution with the right parameter given by the previous generation
(possibly zero). If there is at least one offspring use $\xi_n$ to decide whether it is a mutant and if so give it fitness $F_n$. Then use other newly sampled Bernoulli variables with parameter $\beta$ and fitnesses to decide whether other variables are mutants and if they are decide their fitness. Then surival implies 
$W_n\geq \xi_n F_n$ as required. 
\smallskip

Now $N:=\sum_{i=1}^t \xi_i$ is binomially distributed with parameters $\beta>0$ 
and $t\in\N$.  We infer that, for any fixed $x>1$,
\begin{align*}
\P(W_i \leq  x~\text{ for all } i\le t) 
& \le \P(F_i \xi_i \leq  x~\text{ for all } i\le t) + \P( \text{extinction}) \\
& = \sum_{i=0}^t {t \choose i } \beta^{i} (1-\beta)^{t-i} \P(F \leq x)^{i}+ \P( \text{extinction})\\
& = (\beta\P(F \leq x) +(1-\beta))^t + \P( \text{extinction}).
\end{align*}
Since $\P(F \leq  x) < 1$ and $\beta>0$, we get
\begin{align*}
\P(W_i \leq & \, x~\text{ for all } i)  
 = \lim_{t\uparrow\infty} \P(W_i \leq x~\text{ for all } i\le t)  \\
& \le \lim_{t\uparrow\infty} (\beta\P(F \leq x) +(1-\beta))^t + \P( \text{extinction})
= \P( \text{extinction}),
\end{align*}
hence $\P((W_t)$ is unbounded$) = \P( \text{survival})$ as claimed.
\end{proof}
\medskip

We next describe the distribution of $W_t$ given the process at time $t-1$.

\begin{lemma}
\label{Th:Wa}
Suppose that at generation $t-1$ there are $n$ individuals with
fitness $F_1$, $F_2$, \ldots, $F_n$ and set $\X := \sum_{i=1}^n F_i$. Then,
for all~$x \geq 0$,
\begin{align*}
\P(W_{t}>x )
=1- e^{-\beta \X G(x) }.
\end{align*}
\end{lemma}

\begin{proof}
First fix a positive integer $n$ and suppose $W_t^{_{(n)}}$ is the largest of $n$ independently
sampled fitnesses and  $W^{_{(0)}}_t=0$. Let $\bar G(x)=1-G(x)$ 
and note that
$$\P(W^{_{(n)}}_{t}>x) = 1- \P(W^{_{(n)}}_{t} \leq  x) = 1- \bar G(x)^n.$$
Now let $N$ be the number of
mutants in generation~$t$, which is Poisson distributed with
mean $\beta \X$. Hence, for $x \geq 0$,
\begin{align*}
\P(W_t>x )
& = \sum_{n=0}^\infty \P(W_t^{_{(n)}}>x)\P(N=n)
 = \sum_{n=0}^\infty (1-\bar G(x)^n)  \P(N=n) \\ & = 
1- \sum_{n=0}^\infty \bar G(x)^n
\frac{(\beta \X)^{n}}{n!} e^{-\beta \X}
= 1 -  e^{ -\beta \X(1-\bar G(x))}.
\end{align*}
As $1-\bar G(x)=G(x)$ the proof is complete
\end{proof}

The next two results concern the potential limit~$\nu(\alpha)$. We first characterise $\nu(\alpha)$ 
as a maximum and then as the growth rate in a recursion relation. Note that the first result easily implies that $\nu(\alpha)$
is decreasing, as well as continuous and positive.

\begin{lemma}
\label{lem:nu}
We have
$$\nu(\alpha)=\max\{\tfrac1m \log(m/\alpha) \colon m\in\mathbb N \}.$$
In particular, for all $m \in \N$, we have
\begin{align}
 m \le \alpha e^{ \nu(\alpha) m} \quad \text{ and } \quad
T = \alpha e^{\nu(\alpha) T}.
\label{Eq:numax}
\end{align}
\end{lemma}
\begin{proof}
First observe that $\alpha > m^{m+1}/(m+1)^m$ for $m<T$ and
$\alpha \le m^{m+1}/(m+1)^m$ for $m \ge T$, where $m \in \N$.
Since
$$
 m(m+1)\left (\frac{1}{m} \log \frac{m}{\alpha}  - \frac{1}{m+1}
\log\frac{m+1}{\alpha} \right )  = \log \frac{m^{m+1}}{(m+1)^m \alpha},
$$
we have the desired result.
\end{proof}

\begin{remark}
\label{Rem:alphaT}
Let $\alpha_T := T^{T+1}/(T+1)^T$.
The equality $m = \alpha e^{\nu(\alpha)m}$ holds iff ($m=T$) or ($m=T+1$ and $\alpha = \alpha_T$).
\end{remark}

For the remainder of this subsection, we abbreviate $\nu:=\nu(\alpha)$.

\begin{lemma}\label{ideal_conv}
For some positive sequence $(a_n)$ we define inductively
\begin{equation}
\label{eq:recursion}
\chi_t:=\chi_t(\alpha, (a_n)) := \max\big\{a_t, \tfrac{t-1}{\alpha}\chi_1, \ldots, \tfrac{1}{\alpha}\chi_{t-1} \big\}.
\end{equation}
Then, if $\displaystyle \lim_{n\to\infty} a_n e^{-\nu n}=0$, there are positive constants $c$ and $c'$ such that
$$c' e^{\nu t} \le \chi_t \le c e^{\nu t} \quad \text{  for all $t \ge 1$,}$$ 
and therefore we have
$$\lim_{t\to\infty}\frac{\log \chi_t}{t}=\nu.$$
\end{lemma}

\begin{proof}
Abbreviate $\hat \chi_t := \chi_t /c'$
with $c'=e^{-\nu T}\min\{\chi_1,\chi_2,\ldots,\chi_T\}$.
Obviously, $\hat \chi_t \ge e^{\nu t}$ for $t \le T$.
Now assume $ n \ge T$ and $\hat \chi_t \ge e^{\nu t}$ for all $t \le n$.
By the assumption and~\eqref{Eq:numax}, we have
\begin{align*}
\hat \chi_{n+1} 
\ge
\frac{T}{\alpha} \hat \chi_{n+1-T} 
\ge\frac{T}{\alpha} e^{\nu(n+1-T)}
=e^{\nu(n+1)}.
\end{align*}
Induction gives
$\hat \chi_t \ge e^{\nu t}$ and hence $\chi_t \ge c' e^{\nu t}$ for all $t \ge 1$.
\medskip

Now, choose a positive integer $n_0$ such that $a_n \le e^{\nu n}$ for all $n \ge n_0$. 
Let $\bar \chi_t = \chi_t/c$ with $c = \max\{1,\chi_1,\chi_2,\ldots,\chi_{n_0}\}$.
Obviously,
$\bar \chi_t \le 1 \le e^{\nu t}$ for all $t \le n_0$.
Now let $n \ge n_0$ and assume that $\bar \chi_t \le e^{\nu t}$ for all $t \le n$.
Then,
\begin{align*}
\bar \chi_{n+1} &= 
\max\left \{\tfrac{a_{n+1}}{c}, \tfrac{n}{\alpha}\bar \chi_1,\tfrac{n-1}{\alpha}\bar \chi_2,\ldots,
\tfrac{n-k+1}{\alpha}\bar \chi_k,\ldots,\tfrac{1}{\alpha}\bar \chi_{n}\right \} \\
&\le
\max\left \{e^{\nu(n+1)}, \tfrac{n}{\alpha}e^\nu,\tfrac{n-1}{\alpha}e^{2\nu},\ldots
\tfrac{n-k+1}{\alpha} e^{\nu k},\ldots,\tfrac{1}{\alpha} 
e^{\nu n}\right \} \le e^{\nu(n+1)},
\end{align*}
where we have used~\eqref{Eq:numax}.
By induction, we have $\chi_t \le c e^{\nu t}$ for all $t\ge 1$. 
\end{proof}

\noindent
For later reference we define 
\begin{equation}\label{defi5}
\tilde \chi_i(t) := 
\begin{cases}-\infty, & \text{ if }  i<0,\\
a_{t},&  \text{ if } i=0,\\
(t-i)\chi_i/\alpha, & \text{ if }  1\le i \le t-1.
\end{cases}
\end{equation}

\begin{lemma}
\label{Th:Tp}
Define 
\begin{align}
\label{Eq:It}
I_t:=\max\{i<t \colon \chi_t = \tilde \chi_i(t)\}.
\end{align}
Then $t-I_t$ is bounded. 
\end{lemma}

\begin{proof}
By Lemma~\ref{ideal_conv} we have $c'e^{\nu t} \le \chi_t \le c e^{\nu t}$ for all $t$.
Since there is $t_0$ such that $c'e^{\nu t} > a_t$ for all $t \ge t_0$,
we can write, for $t > t_0$,
\begin{align*}
\chi_t = \max\left \{\tfrac{t-1}{\alpha}\chi_1,\ldots,\tfrac{1}{\alpha}\chi_{t-1}
\right \}.
\end{align*}
Now it is enough to show that $t - I_t$ is bounded for $t> t_0$.
\smallskip

Note that, for $1 \le m \le t-1$,
\begin{align*}
c e^{\nu t} A(m)
\ge \tfrac{m}{\alpha} \chi_{t-m} \ge 
c' e^{\nu t} A(m),
\end{align*}
where $A(m) = me^{-\nu m} /\alpha $ with $A(T)=1$.
Since $\lim_{m\rightarrow \infty} A(m)=0$, there is $m_0$ such that
$c'> c A(m)$ and hence 
$$\tfrac{m}{\alpha} \chi_{t-m} < c' e^{\nu t},  \text{ for all $m \ge m_0$.}$$
As the right hand side is a lower bound of $\frac{T}\alpha \chi_{t-T}$ we get that
$t - I_t$ cannot be larger than $\max\{m_0,t_0\}$, as desired.
\end{proof}

\begin{remark}
\label{Re:Tp}
If we choose $T' > \sup\{t-I_t \colon t\in\N\}$, then, for all $t$,
\begin{align*}
\chi_t = \max\{a_t,(t-1)\chi_1/\alpha,\ldots,\chi_{t-1}/\alpha\}
=\max\{\tilde\chi_{t-T'}(t),\ldots,\tilde \chi_{t-1}(t)\}.
\end{align*}
In words, $\chi_t$ is completely determined
by $\tilde \chi_i(t)$ for $i$ within the window $t-T' \le i \le t-1$.
This fact will play an important role in the proof of Theorem~\ref{Th:mainthm}.
\end{remark}

\noindent
We conclude the subsection with two estimates for classical Galton-Watson processes.

\begin{lemma}
\label{Th:galton}
Consider a supercritical Galton-Watson process $(\X_t)_{t\geq0}$ with Poisson offspring distribution with mean $\theta>1$,
starting in generation 0 with a single 
individual. Fix $0<x<1$ and an integer $n \ge 1$. Then,
\begin{align}
\label{Eq:GL}
\P\left ( \X_t \ge   x^t \theta^t \text{ for all } 1\le t \le n  \right ) 
\ge 1 
- n \frac{(1- x)^{-2}}{\theta-1}.
\end{align}
\end{lemma}

\begin{proof}
First note that 
(see, e.g., \cite{H63})
\begin{align*}
\E(\X_t) = \theta^t,\quad \V(\X_t) = \frac{\theta^{2t} - \theta^{t}}{\theta-1},
\end{align*}
and that
\begin{align}
\P\left ( \X_t \ge    x^t \theta^t \text{ for all }1\le t \le n \right ) 
\ge 1 - \sum_{t=1}^{n} 
\P\left ( \X_t \le  x^t \theta^t \right ),
\label{Eq:GWl}
\end{align}
where we have used the sub-additivity of probability measure.
Using Chebyshev's inequality, we get
\begin{align*}
\P&\left ( \X_t \le  x^t \theta^t \right )=
\P\left (  \theta^t - \X_t \ge   \theta^t-x^t \theta^t \right ) \le
\P\left ( \vert \X_t- \theta^t\vert  \ge   \theta^t-x^t \theta^t \right ) \\
&\le \left (1-x^t \right )^{-2} \frac{\V(\X_t)}{ \theta^{2t}} 
=\left (1-x^t \right )^{-2}\frac{(1-\theta^{-t})}{(\theta-1)} 
\le \frac{(1-x)^{-2}}{\theta-1},
\end{align*} 
which, along with~\eqref{Eq:GWl}, gives the claimed inequality.
\end{proof}

\begin{lemma}
\label{Th:tdG}
For a Galton-Watson process $(\X_t)$ with $\X_0 = K_0$ and generation dependent 
offspring distribution $N_t$ with $\E [N_t] \le N$ for all $t$,
\begin{align*}
\P\big(\X_t \le K N^t B^t \text{ for all } t\ge 1\big) 
\ge 1 - \frac{K_0}{K(B-1)},
\end{align*}
for all $B>1$ and $K>0$.
\end{lemma}

\begin{proof}
By Markov's inequality, we have
\begin{align*}
\P(\X_t \ge K N^t B^t ) \le \frac{\E[\X_t]}{K N^t} B^{-t}.
\end{align*}
Since $\E[\X_{t+1}\vert \X_t] = \X_t \E [N_t]$,
we have $\E[\X_t] = K_0 \prod_{i=0}^{t-1} \E [N_i] \le K_0 N^t$,
which gives
\begin{align*}
\P(\X_t \ge K N^t B^t ) \le \frac{K_0}{KB^{t}}.
\end{align*}
Since
$$
\P( \X_t \le K N^t B^t \text{ for all }t\ge1) 
\ge
1-\sum_{t=1}^\infty
\P(\X_t \ge K N^t B^t) ,
$$
a geometric sum gives the claimed inequality.
\end{proof}

\subsection{Proof of the lower bound}\label{low}

In this subsection we show that, for given $\alpha>0$ and all $\alpha'>\alpha$, we have 
\begin{align}
\P\Big( \liminf_{t\rightarrow\infty}\frac{\log \log X(t)}{t} \ge \nu(\alpha') \,\Big\vert \, \text{survival}\Big) = 1.
\label{Eq:main1}
\end{align}
{In both models at each generation $s$ the lineage originating from
the mutant with fitness $W_s$ dominates a version  $(\hXt(f))_{t\geq s}$  of the 
same model starting in generation~$s$ with a single individual of fitness $f=W_s$.}
If there is at least one~$s$ such that
$$
\liminf_{t\rightarrow\infty}\frac{\log \log \hXt(W_s)}{t} \ge \nu(\alpha')
$$
then~\eqref{Eq:main1} is proved. As $(W_t)$ is unbounded almost surely on survival
it therefore suffices to show that
\begin{align}
\lim_{f\rightarrow \infty} 
\P\left ( 
\liminf_{t\rightarrow\infty}\frac{\log \log \hXt(f)}{t} \ge \nu(\alpha')\right 
) = 1.
\label{Eq:mainl}
\end{align}
As $(\hXt(x))$ can be coupled to an \sm $(\St(x))$ with the same initial condition such that 
$\hXt(x)\geq \St(x)$ for all $t\geq0$, the result follows by combining Lemma~\ref{ideal_conv} with the following statement.

\begin{lemma}
\label{Th:mainlow}
Fix  $0<\epsilon<1/2$ and let 
$$E(f):= \big\{ \St(f) \ge (1-\beta)^t f^{\chi_t'}   \text{ for all } t\ge 1  \big\},$$
where 
$\chi_t' := \chi_t(\alpha', (\frac{n}2))$ with $\alpha':=\alpha/(1-2 \epsilon)$.
Then 
$$
\lim_{f \rightarrow \infty} \P(E(f)) =1.
$$
\end{lemma}

\begin{proof}
We define
$\w_0:=f$,
$\n_0:=f^{1/2}$,
and ($i \ge 1$)
\begin{align*}
\w_i:=f^{(1-\epsilon)\chi_i'/\alpha},\quad
\n_i:=f^{(1-2\epsilon)\chi_i'/\alpha}=f^{\chi_i'/\alpha'}.
\end{align*}
For later reference, we also define $U_i := \n_i/\w_i$ for all $i\ge 0$.
\smallskip

Set $\epsilon_0=\epsilon/(2-2\epsilon)$. 
By our assumption on $\mu$, there is $f_0$ such that
$$G(x) \ge x^{-\alpha(1+\epsilon_0)} \text{ for all $x > f_0$.}$$
Since we are only interested in the limit as $f\to\infty$,
we may assume that $f$ is so large that
$(1-\beta)\w_0 \ge 2$, $(1-\beta) \w_1 \ge 2$, $U_0 < 1/2$,
$U_1 < 1/2$, and $\w_1 > f_0$. Notice that by assumption,
$$G(\w_1)\ge G(\w_i) \ge \w_i^{-\alpha(1+\epsilon_0)} = f^{-(1-\epsilon/2)\chi_i'}
\text{ for all $i \ge 1$.}$$
For $\chi'_t$, we choose $T'$ as in Remark~\ref{Re:Tp}. 
By $N_{i,t}$ we denote 
the number of individuals with fitness $W_i$ at generation $t$.
Define events 
\begin{align*}
A_i:= \{ N_{i,t} \ge (1-\beta)^{t-i} \n_i^{t-i} \text{ for all }i < t \le i+T'  \}, \quad 
B_i:= \{ W_i \ge \w_i \}.
\end{align*}
Let $D_{-1}$ be the certain event and, for $i\in\N_0$,
$$
D_{i}:= D_{i-1} \cap A_i \cap B_i,\quad A := \bigcap_{i=0}^\infty D_i.
$$
Now observe that
\begin{align*}
\P(A) = \lim_{i\rightarrow\infty} \P(D_i),\quad
\P(D_i) = \P(A_i\vert B_{i}\cap D_{i-1}) \P(B_i\vert D_{i-1})\P(D_{i-1}).
\end{align*}
By Lemma~\ref{Th:galton} we have
\begin{align*}
\P(A_i \vert  B_i \cap D_{i-1}) 
\ge 1 - 
\frac{T'(1-\n_i/W_i)^{-2}}{(1-\beta) W_i -1}
\ge 1 -\frac{T'(1-U_i)^{-2}}{(1-\beta) \w_i -1},
\end{align*}
where we have used~\eqref{Eq:GL}. 
\smallskip

To proceed, we find the $\X$ in Lemma~\ref{Th:Wa} on the event $D_{n-1}$
as 
\begin{align*}
\X \ge \sum_{i=0}^{t-1} N_{i,t-1} W_i
\ge (1-\beta)^t \sum_{i=t-T'-1}^{t-1} f^{\tilde\chi_i'(t)}
\ge(1-\beta)^t  f^{\chi_{t}'},
\end{align*}
where we have used $W_i \ge \w_i \ge \n_i$ and $\tilde \chi_i'(t)$ as in \eqref{defi5} for parameters $\alpha'$
and $a_s=s/2$. Using Lemma~\ref{Th:Wa} 
with $G(\w_i) \ge f^{-(1-\epsilon/2)\chi_i'}$, we have
\begin{align*}
\P(B_i\vert D_{i-1}) 
\ge 
1- \exp \left ( - \beta(1-\beta)^i  f^{\epsilon\chi_i'/2} \right ).
\end{align*}
Now we define
\begin{align*}
b_i:= \frac{T'(1 - U_i)^{-2}}{(1-\beta) \w_i-1} + 
(1-\delta_{i,0})\exp \left ( - \beta(1-\beta)^i  f^{\epsilon\chi_i'/2} \right ),
\quad
\g(f) := \sum_{i=0}^\infty b_i,
\end{align*}
where $\delta_{ij}$ is the Kronecker delta symbol.
Trivially, we have
$\lim_{f\rightarrow \infty} b_i = 0$ for all $i\ge 0$.
Since, for sufficiently large $f$, 
$b_s$ for fixed $s$ is a bounded and decreasing function of $f$ and 
since Lemma~\ref{ideal_conv} gives
$$
\lim_{s\rightarrow \infty} b_s 2^s = 0,
$$
there is $s_0$ such that $\vert b_s\vert  < 2^{-s}$ for
all $s >s_0$ and for all assumed value of $f$. Therefore, the series defining $\g(f)$ 
converges uniformly for sufficiently large $f$
and $\lim_{f \rightarrow \infty} \g(f)=0$. 
Therefore, for sufficiently large $f$, we get
\begin{align*}
\P(A) \ge \prod_{i=0}^\infty (1-b_i)
\ge
1-\g(f),\quad
\lim_{f\rightarrow \infty} \P(A) = 1,
\end{align*}
where we have used $(1-x)(1-y) \ge 1- x - y$ for $x,y\ge 0$.
As, on the event $A$,
\begin{align*}
\St(f) \ge \sum_{i=t-T'}^t N_{i,t} 
\ge (1-\beta)^t f^{\chi_t'}
\end{align*}
where we have assumed $N_{i,t} =0$ for $i<0$, 
we see that $A\subset E(f)$ and the proof is completed.
\end{proof}

In fact, Lemma~\ref{Th:mainlow} and its proof are applicable to the \mm verbatim,
except that $\St$ is replaced by $\hXt$. If we are interested in the proof only
for the \mm, we actually do not need to introduce $\St$.

\subsection{Proof of the upper bound}\label{up}

In this subsection we show that, for given $\alpha>0$ and all $\alpha'<\alpha$, we have for the
\mm denoted by $(\Ct)$, or $(\Ct(x))$ if in the initial generation there is a single individual with fixed fitness $x$,
that
\begin{align}
\P\Big( 
\limsup_{t\rightarrow\infty}\frac{\log \log \Ct}{t} \le \nu(\alpha')\Big)= 1.
\label{Eq:main2}
\end{align}
In case of extinction the upper bound holds by convention.
One can construct two processes with initial fitness $x\leq y$ on the same probability 
space such that $\Ct(x) \le \Ct(y)$ for all~$t$. Indeed, this can be done as follows. First construct $(\Ct(y))$ and look
at its genealogical tree truncated after the first mutant in every line of descent from the root. Removing any individual in that
tree together with all its offspring from $(\Ct(y))$ independently with probability $x/y$ we obtain $(\Ct(x))$.
\medskip%

We now construct an \mm  with special initial conditions. 
Fix $\epsilon>0$. For given~$\alpha$, let $\delta = (1+2\epsilon)/(1+3\epsilon)$, $\alpha'=\alpha/(1+3\epsilon)$, $\nu'=\nu(\alpha')$,
$T = T(\alpha')$, and
$$\Delta_0 = \frac{\epsilon}{\nu' (1+3\epsilon)},$$ 
which is equivalent to 
$\nu' \Delta_0 +\delta = 1$.
We choose $\Delta$ such that $0 < \Delta \le \Delta_0$ and
\begin{align*}
\hat n:=\frac{T}{\Delta}
\end{align*} 
is an integer.
We define, for a given $f>0$,
\begin{align*}
\chi_t' := e^{ \nu't},\quad \kappa_{n,t} := e^{\nu'(t-T + n \Delta)},\quad
\G_{n,t} := f^{\kappa_{n,t}/\alpha},
\quad
h_{n,t} := f^{(1+\epsilon)\kappa_{n,t}/\alpha}.
\end{align*}
We consider the \mm $(\wCt(f))_{t\geq T-1}$
starting in generation $T-1$ with an initial condition
such that 
there are $T$ different mutant classes with fitness $\G_{\hat n,m}$ 
for $0\leq m\leq T-1$
and the number of individuals with fitness $\G_{\hat n,m}$
is $\lfloor (\G_{\hat n,m})^{T-m-1} \rfloor$.
We only consider $f$ sufficiently large so that $(1-\beta) f>2$ and $(1-\beta) f^{\epsilon/\alpha}>2$.
\medskip

Now assume that we have proved, for all $\alpha'< \alpha$,
\begin{align}
\lim_{f\rightarrow\infty} \P
\Big( \limsup_{t\rightarrow \infty} \frac{\log \log \wCt(f)}{t} 
\le \nu(\alpha')\Big) =1.
\label{Eq:tbd}
\end{align}
Given an arbitrary $f>0$ and $\varepsilon>0$ pick $f_\varepsilon$ such that the probability above exceeds $1-\varepsilon$
and the smallest fitness in the initial condition of 
$\wCt(f_\varepsilon)$ is larger than $f$.  
Then \eqref{Eq:tbd} guarantees that
$$\P\Big(  \limsup_{t\rightarrow \infty} \frac{\log \log \Ct(f)}{t} 
\le \nu(\alpha') \Big) >1-\varepsilon,$$
which 
proves~\eqref{Eq:main2}.
So it is enough to prove~\eqref{Eq:tbd}.
Once~\eqref{Eq:tbd} is proved,
we use the natural coupling such that $\St \leq \Ct$ for all $t$. 
Then, almost surely on survival,
\begin{align*}
 \limsup_{t\rightarrow \infty} \frac{\log \log S_t}{t}
\le \limsup_{t\rightarrow \infty} \frac{\log \log M_t}{t}\le \nu(\alpha),
\end{align*}
which completes the proof of Theorem~\ref{Th:mainthm}.

\begin{lemma}
\label{Th:nonmut}
Let $Z_t$ be the number of \emph{non-mutated}
descendants at generation~$t$ of $\X$ individuals with fitness in a bounded 
interval $I$ with right endpoint $b$ at generation~$m$ of an \mm. Assume $\X \le K$. 
Then, for all $B>1$,
$$
\P (Z_t \le K (1-\beta)^{t-m} b^{t-m} B^{t-m} \text{ for all } t\ge m+1) \ge 1 - \frac{1}{B-1}.
$$
\end{lemma}

\begin{proof}
As the mean number of non-mutated offspring of an individual is bounded by
$(1-\beta)b$ we get the result by applying Lemma~\ref{Th:tdG}.
\end{proof}

\begin{lemma}
\label{Th:thmmut}
Suppose at generation $t-1$ of an \mm the population consists of $n$ individuals with
fitness $F_1,\ldots, F_n$.  Let 
\begin{align}
Y_{t}:= \beta \sum_{i=1}^n F_i
\label{Eq:bY}
\end{align}
and
let $Z$ be the number of mutants in generation $t$ 
with fitness in the interval~$(a,b]$. 
Then, with $p := \mu((a,b])$, we have
\begin{align*}
\P(Z> K)
=e^{-Y_tp}\sum_{n=K+1}^{\infty} \frac{(Y_tp)^n}{n!}. 
\end{align*}
\end{lemma}

\begin{proof}
Observe that $Z$ is Poisson distributed with mean $Y_tp$.
\end{proof}

\begin{remark}
\label{Re:Zm}
Using Markov's inequality, we get 
\begin{align}
\label{Eq:ZmM}
\P(Z > K) \le \frac{Y_t p}{K},
\end{align}
which is useful when $K \gg Y_tp$.
By Chebyshev's inequality, for $K > Y_t$,
\begin{align}
\label{Eq:ZmC}
\P(Z > K ) 
\le \P(\vert Z-Y_tp\vert  \ge K - Y_tp)
\le \frac{Y_tp}{(K-Y_tp)^2} \le \frac{Y_t}{(K-Y_t)^2},
\end{align}
which is useful when $(K-Y_t)^2 \gg Y_t$.
For $K=0$, we will use
\begin{align}
\label{Eq:Z0}
\P(Z= 0) = e^{-Y_tp} 
\ge 1-Y_tp\ge 1- Y_t G(a).
\end{align}
\end{remark}

We denote the number of non-mutated descendants of initial
individuals with fitness $\G_{\hat n,m}$ 
at generation $t\ge T-1$ by $N_{m,T-1,t}$ 
and define $$N_{T-1,t} := \sum_{m=0}^{T-1} N_{m,T-1,t}.$$
The number of mutants that appear at generation $t\geq T$ with fitness 
in the interval $(h_{n-1,t}, h_{n,t}]$ is denoted by $N_{n,t,t}$ for $0 \le n \le \hat n+1$, where
we have assumed $h_{-1,t}:=0$ and $h_{\hat n+1,t}:=\infty$.
Typically, $N_{\hat n+1,t,t}$ will be zero.
%
The number of non-mutated descendants of $N_{n,m,m}$ at generation
$t>m$ is denoted by $N_{n,m,t}$. For $t \ge m \ge T$ define 
\begin{align*}
N_{m,t} := \sum_{n=0}^{\hat n+1} N_{n,m,t},
\end{align*}
which gives
\begin{align*}
\wCt(f) = \sum_{m=T-1}^t N_{m,t}.
\end{align*}
Let $(\theta_t)_{t\geq T-1}$ be a sequence satisfying
$\theta_{T-1} = \theta_{T} = T$ and, for $t>T$,
$$\theta_t = \hat n + \sum_{m=T-1}^{t-1} \theta_m.$$ 
Since $\theta_{t+1} - \theta_t = \theta_t$ for $t \ge T+1$, we have
\begin{align}
\theta_t = 
\begin{cases} 2^{t-T-1}(2 T + \hat n), & \text{ for }t > T\\
 T, & \text{ for }t\le T.
\end{cases}
\label{Eq:theta}
\end{align}

\begin{lemma}
\label{Lem:xl}
For  $T\le x\le m<t$ ($t,m$ are integers
and $x$ is real), we have
\begin{align}
\label{Eq:xl}
&e^{\nu'm } + \tfrac{t-m}{\alpha'} e^{\nu' (m-T)} \le e^{\nu' t},\\
\label{Eq:xu}
&e^{\nu'm} - e^{\nu'(x-\Delta)} +\delta \tfrac{t-m}{\alpha'} e^{\nu' x} \le e^{\nu' t}.
\end{align}
\end{lemma}

\begin{proof}
Using~\eqref{Eq:numax}, we have
\begin{align*}
e^{\nu'm} \left ( 1 + \tfrac{t-m}{\alpha'} e^{-\nu' T} \right )
= e^{\nu'm } \tfrac{T+t-m}{T} 
\le e^{\nu' t} e^{\nu' T} \tfrac{\alpha'}{T}
= e^{\nu' t},
\end{align*}
which proves~\eqref{Eq:xl}.
If $\delta (t-m) - \alpha' e^{-\nu' \Delta}$ is negative, then~\eqref{Eq:xu} is trivially valid.
If $\delta (t-m) - \alpha' e^{-\nu' \Delta}$ is positive, then
the left hand side of~\eqref{Eq:xu} has maximum at $x=m$. 
Therefore, it is enough to prove~\eqref{Eq:xu} only for $x =m$.
Plugging $x=m$, we have
\begin{align*}
e^{\nu'm} - e^{\nu'(m-\Delta)} + \delta \tfrac{t-m}{\alpha'} e^{\nu' m} 
& =e^{\nu'm} ( 1 - e^{-\nu'\Delta} ) + \delta
\tfrac{t-m}{\alpha'} e^{\nu' m} \\ &
\le e^{\nu' t} \left ( \nu' \Delta + \delta
\right ) \le 
e^{\nu' t},
\end{align*}
where we have used $1-e^{-x} \le x$, $e^{\nu'm} \le e^{\nu't}$, and $t-m \le \alpha' e^{\nu'(t-m)}$.
\end{proof}

\begin{lemma}
\label{Th:mainup}
Let $E(f):=\{\wCt(f) \le  \theta_{t+1} f^{\chi'_t}$ for all $t \ge T \}$.
Then
$$
\lim_{f\rightarrow \infty} \P(E(f)) = 1,
$$
which implies~\eqref{Eq:tbd}.
\end{lemma}

\begin{proof}
Set $\epsilon_0=\epsilon/(2+2\epsilon)$. By our assumption on $\mu$
there is $f_0$ such that 
$G(f) \le f^{-\alpha(1-\epsilon_0)}$ for all $f > f_0$.
Now we assume $h_{0,T}=f^{(1+\epsilon)/\alpha} > f_0$,
which gives
\begin{align}
G(h_{n,t}) \le f^{-(1+\epsilon/2)\kappa_{n,t}}, \quad \text{for all } 0\le n \le
\hat n \text{ and } t \ge T.
\end{align}
Let $$A_{n,T-1}:=\left \{ N_{n,T-1,t} \le \left \lfloor (\G_{\hat n,n})^{T-n-1} \right \rfloor 
(1-\beta)^{t-T+1} 
f^{(t-T+1)\chi'_n/\alpha'} \mbox{ for all } t\ge T \right \} $$
and define $A_{T-1} := \bigcap_{n=0}^{T-1} A_{n,T-1}$.
By Lemma~\ref{Th:nonmut} with $K= \lfloor (\G_{\hat n,n})^{T-n-1}  \rfloor$, $b = \G_{\hat n,n}=f^{\chi_n'/\alpha}$, and
$b B = f^{\chi_n'/\alpha'} = f^{(1+3\epsilon)\chi_n'/\alpha}$, we have
\begin{align*}
1-\P(A_{n,T-1}) \le \left ( f^{3 \epsilon \chi_n'/\alpha}-1\right )^{-1}, & \quad
\P(A_{T-1}) \ge 1 - b_{T-1}, \\
b_{T-1}:= \sum_{n=0}^{T-1} \left ( f^{3 \epsilon \chi_n'/\alpha}-1\right )^{-1}.
\end{align*}
Since $m e^{-\nu' m} \le \alpha'$ and $\G_{\hat n,n} \le f^{\chi'_n/\alpha'}$, 
we have $$N_{n,T-1,t} \le f^{\chi_n'(t-n)/\alpha'} \le f^{\chi'_t}\quad \text{for all }n,t.$$
Therefore, 
\begin{align}
\sum_{n=0}^{T-1}N_{n,T-1,t} \le T f^{\chi'_t},\quad \text{for all }t \ge T,
\label{Eq:NAT}
\end{align}
on the event $A_{T-1}$.
Let 
$$
A_{n,m,t} :=
\left\{
\begin{array}{ll}
\big\{ N_{0,m,t} \le 
\theta_m f^{\chi_m'} (1-\beta)^{t-m} 
f^{(t-m)\kappa_{0,m}/\alpha'} \big\},
&\quad \text{ for }n=0,\\[2mm]
\big\{N_{n,m,t} \le \frac{f^{\chi_n'}}{(\G_{n-1,m})^\alpha} (1-\beta)^{t-m} 
f^{\delta (t-m) \kappa_{n,m}/\alpha'} \big\},
&\quad \text{ for }1\le n\le \hat n, \\[3mm]
\big\{ N_{\hat n+1,m,t} = 0 \big\} &\quad \text{ for }n = \hat n +1.
\end{array}
\right.$$
Note that $A_{n,t,t}$ has information on the empirical distribution of 
mutants' fitness that appear at generation $t$.
We define
\begin{align*}
&\tilde A_{n,m} = \bigcap_{t=m+1}^\infty A_{n,m,t},\quad
A_{n,m}:=A_{n,m,m} \cap \tilde A_{n,m},\quad
A_m := \bigcap_{n=0}^{\hat n+1} A_{n,m},\\
&D_{T-1} := A_{T-1},\quad D_{n} := D_{n-1} \cap A_n,\quad
A(f) := \bigcap_{n=T-1}^\infty D_n.
\end{align*}
By Lemma~\ref{Lem:xl}, we have on the event $A_m$, that for all $t\ge m$,
\begin{align*}
\sum_{n=0}^{\hat n+1} N_{n,m,t} \le (\theta_m + \hat n) f^{\chi'_t},
\end{align*}
and, in turn,
\begin{align*}
\wCt(f) \le \left (\hat n + \sum_{m=T-1}^ t \theta_m \right ) f^{\chi'_t}
=\theta_{t+1} f^{\chi'_t}
\end{align*}
on the event $A(f)$. Therefore, $A(f) \subset E(f)$ and 
the proof is complete if we show
\begin{align*}
\lim_{f\rightarrow \infty} \P(A(f)) = 1.
\end{align*}
Now we investigate $\P(A_m \vert  D_{m-1})$.
First note that
\begin{align*}
1-\P(A_m\vert D_{m-1}) & \le \sum_{n=0}^{\hat n+1} \left [ 1 - \P(A_{n,m}\vert D_{m-1})
\right ],\\
\P(A_{n,m}\vert D_{m-1}) & =\P(\tilde A_{n,m}\vert A_{n,m,m}\cap D_{m-1})
\P(A_{n,m,m}\vert D_{m-1}),
\end{align*}
and, on the event $D_{m-1}$,
\begin{align}
Y_m\le \beta \theta_m f^{\chi'_m},
\label{Eq:Yt}
\end{align}
where $Y_m$ is defined in~\eqref{Eq:bY}.
\medskip

We begin with $\P(A_{\hat n+1,m}\vert D_{m-1})$, which clearly
equals $\P(A_{\hat n+1,m,m}\vert D_{m-1})$.
Using~\eqref{Eq:Z0} with $a = h_{\hat n,m}$ and 
$Y_m G(a)  \le \theta_m f^{-\epsilon \chi'_m/2}$,
we obtain
\begin{align*}
\P(A_{\hat n+1,m}\vert D_{m-1}) \ge 1 - 
b_{\hat n+1,m},\quad b_{\hat n+1,m}:= \theta_m f^{-\epsilon \chi'_m/2}.
\end{align*}
Now we consider $\P(A_{0,m}\vert D_{m-1})$.
Using~\eqref{Eq:ZmC} with $K = \theta_m f^{\chi_m'}$
and~\eqref{Eq:Yt}, we have
\begin{align*}
\P(A_{0,m,m}\vert D_{m-1}) 
\ge 1- \frac{\beta}{(1-\beta)^2 \theta_{m} } f^{-\chi'_m}.
\end{align*}
Using Lemma~\ref{Th:nonmut} with $B = f^{\kappa_{0,m}/\alpha'}/h_{0,m} =  f^{2\epsilon \kappa_{0,m}/\alpha}$, 
we have
\begin{align*}
\P(\tilde A_{0,m}\vert A_{0,m,m}\cap D_{m-1}) 
\ge 1 - \left ( f^{2 \epsilon \kappa_{0,m}/\alpha} -1 \right )^{-1}.
\end{align*}
Therefore,
defining
$$
b_{0,m}:= \frac{\beta}{(1-\beta)^2 \theta_{m} } f^{-\chi'_m}
+ \left ( f^{2 \epsilon \kappa_{0,m}/\alpha} -1 \right )^{-1},$$
we have
$\P(A_{0,m}\vert D_{m-1}) \ge 1-b_{0,m}$.
\pagebreak[3]
\medskip

Finally, we move on to 
$\P(A_{n,m}\vert D_{m-1})$ for $1 \le n \le \hat n$.
Using~\eqref{Eq:ZmM} with $K = f^{\chi_m' - \kappa_{n-1,m}}$,
$a = h_{n-1,m}$,
and $Y_m G(a) \le \theta_m f^{\chi_m' - (1+\epsilon/2) \kappa_{n-1,m}}$, we have
\begin{align*}
\P(A_{n,m,m}\vert D_{m-1}) \ge 1 -  \theta_m f^{-\epsilon \kappa_{n-1,m}/2}.
\end{align*}
Using Lemma~\ref{Th:nonmut} with $B = f^{\delta \kappa_{n,m}/\alpha'}/h_{n,m} =  f^{\epsilon \kappa_{n,m}/\alpha}$, 
we have
\begin{align*}
\P(\tilde A_{n,m}\vert A_{n,m,m}\cap D_{m-1}) 
\ge 1 - \left ( f^{\epsilon \kappa_{n,m}/\alpha} -1 \right )^{-1}.
\end{align*}
Therefore,
defining
$$
b_{n,m}:= \theta_m f^{-\epsilon \kappa_{n-1,m}/2}
+\left ( f^{\epsilon \kappa_{n,m}/\alpha} -1 \right )^{-1},
$$
we have
$
\P(A_{n,m}\vert D_{m-1}) \ge 1-b_{n,m}.
$
We define  
$$b_m:= \sum_{n=0}^{\hat n+1} b_{n,m} \text{ for }m \ge T,\quad \text{ and }
\g(f) := \sum_{m=T-1}^\infty b_m.$$
Recall that we have assumed $(1-\beta) f>2$ and $(1-\beta) f^{\epsilon/\alpha}>2$.
Since $b_m$ for given $m$ is a bounded function 
of $f$ which is decreasing to zero and 
$$
\lim_{m\rightarrow \infty} b_m 2^{m} = 0,
$$ 
there is $m_0$ such that $\vert  b_m\vert  < 2^{-m}$
for all $m >m_0$. Hence the series defining $\g(f)$ converges 
uniformly for sufficiently large $f$ 
and, accordingly, $\lim_{f \rightarrow \infty} \g(f) = 0$. 
Therefore, for sufficiently large $f$,
\begin{align*}
\P(A(f)) \ge \prod_{m=T-1}^\infty
( 1 - b_m) 
\ge
1-\g(f),
\end{align*}
and $\displaystyle\lim_{f\rightarrow \infty} \P(A(f)) = 1$,
which completes the proof.
\end{proof}

\section{\label{Sec:steady}Empirical frequency distributions} 

Apart from the fact that the population is dominated by a single mutant class at all times, the proof of the double-exponential growth rate $\nu$ presented in Sec.~\ref{Sec:pf} does not give any insight into the structure of the population. However, since the solution $\chi_t$ of the recursion relation (\ref{eq:recursion}) 
correctly describes the asymptotic growth of $X(t)$, it provides a natural starting point for addressing this question at least on a heuristic level.
In this section, we analyze the recursion relation in more depth to understand the 
demographic structure in the long time limit, which turns out to display a rather
rich behaviour.\medskip

\subsection{Numerical solution of the recursion relation}
\label{Sec:numerical_recursion}

To characterize the empirical frequency distribution we introduce the following quantities: 
\begin{align*}
J_i(t) &:=\frac{\log W_i}{\log W_t} \approx \frac{\chi_i}{\chi_t}, 
\;\; P(t) := \frac{\log X(t+1)-\log X(t)}{\log W_t} 
\approx \alpha \frac{\chi_{t+1} - \chi_t}{\chi_t}, \\
R_i(t) &:= \frac{\log W_i^{t-i} - \log X(t)}{\log X(t)}= \frac{t-i}{\alpha}J_i(t) -1,
\end{align*}
where the second approximate relations in the definitions of $J_i(t)$ and $P(t)$ become equalities in the formal deterministic limit $f \to \infty$ (see Sec.~\ref{Sec:sim}).
The ratio $J_i(t) \in [0,1]$ compares the log-fitness of the mutant class born at time $i$ to the log-fitness of the current fittest mutant.
Since $X(t+1) \approx (1-\beta) X(t) \bar F_t$ with $\bar F_t$ denoting the mean fitness 
of the population at generation $t$,
$P(t)$ quantifies the mean fitness at generation $t$ on the same scale. The decomposition in Eq.~(\ref{Eq:evol}) shows that the fraction of the population in mutant class $i$ at time $t$ is proportional to $W_i^{t-i}$, and therefore $R_i(t) \in [-1,0]$ serves as proxy of the (logarithmic) frequency distribution at generation $t$ over mutant classes $i$. 
\medskip

\begin{figure}
\centering
\includegraphics[width=\linewidth]{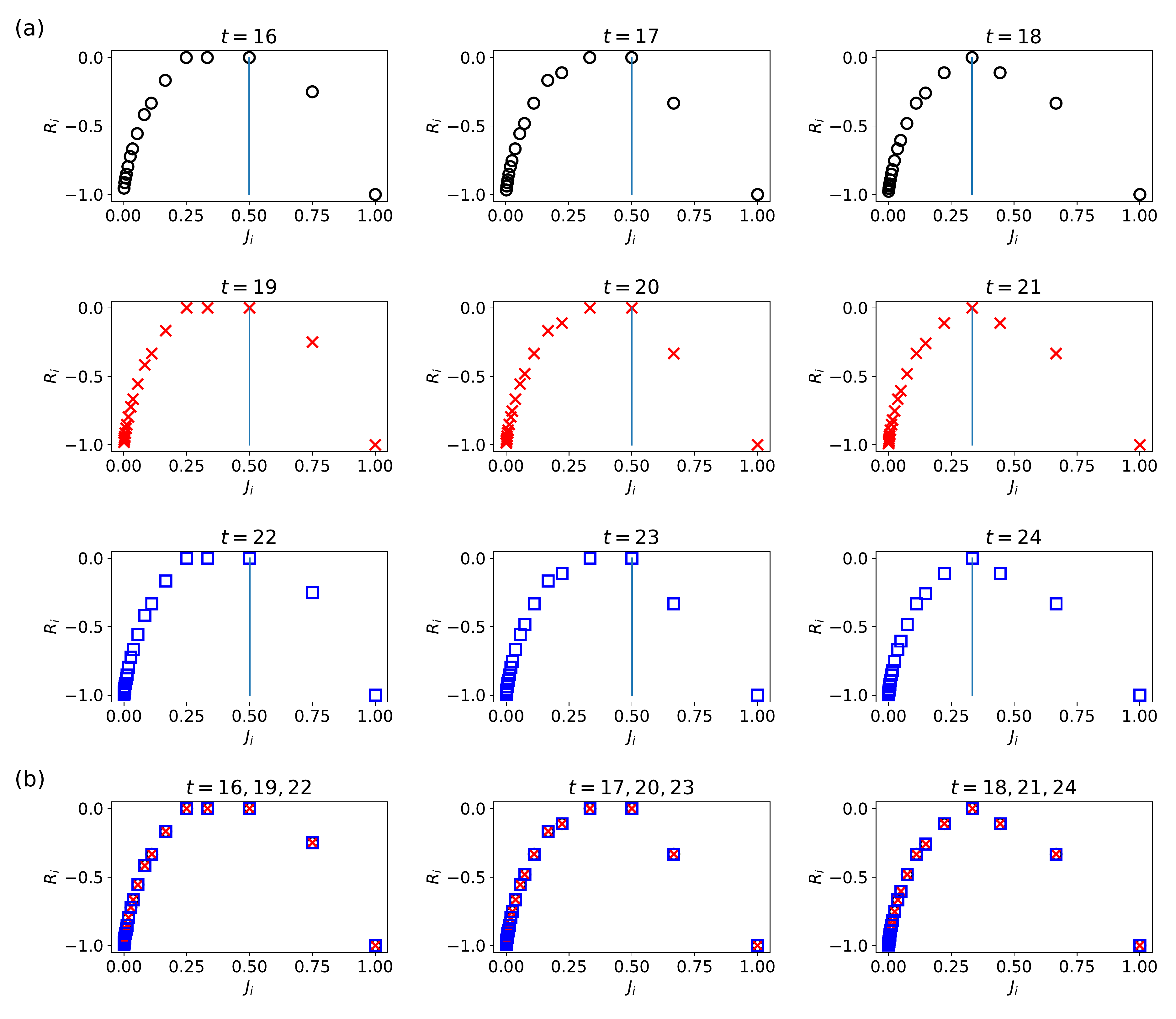}
\caption{\label{Fig:1conf} (a) Plots of $R_i(t)$ vs. $J_i(t)$ 
for $\alpha = 1$ at generations $t=16$, $17$, $\ldots$, $24$.
The vertical line at each panel indicates the location of the mean fitness $P(t)$.
(b) Data collapse plots of $R_i(t)$ vs. $J_i(t)$ for each column of (a).
}
\end{figure}

In Fig.~\ref{Fig:1conf}(a), we plot $R_i(t)$ against $J_i(t)$ for nine consecutive
generations, obtained by numerically solving the recursion relation (\ref{Eq:chire}) for $\alpha = 1$. The salient feature is the periodic behavior 
of the frequency distribution
with period $3$; note that $T=3$ for $\alpha=1$.
To illustrate the accuracy of the periodic behaviour, we present
data-collapse plots in Fig.~\ref{Fig:1conf}(b). In most regions of $J_i$,
the collapse looks perfect (since the number of mutant classes increases with the number of generations, the frequency distributions at different times cannot be identical).
A rigorous proof of the periodicity will be given in Sec.~\ref{Sec:chipf}.
\medskip

The periodicity was taken into account in the numerical estimates of $\nu$ reported in Fig.~\ref{Fig:makea}. Rather than monitoring $\log \chi_t/t$, which converges very slowly,
we computed the quantity
\begin{align*}
\hat \nu(t):= \frac{1}{T} \log \frac{\chi_{t+T}}{\chi_{t}},
\end{align*}
which approaches a constant in a relatively short time. 
\medskip

\subsection{\label{Sec:chipf}Periodicity of $\chi_t e^{-\nu t}$}

By Lemma~\ref{ideal_conv}, we know that 
\begin{align*}
c_t := \chi_t e^{-\nu t}
\end{align*}
is bounded away from zero and infinity.  Now we show that $c_t$ is not only bounded, but eventually becomes periodic. 
\begin{proposition}
\label{Prop:periodic}
For any sequence $(a_n)$ in the recursion relation \eqref{eq:recursion}, there is a $t_1$ such that
$c_t = c_{t+T}$ for all $t \ge t_1$.
\end{proposition}
\ \\[-1.5cm]

\begin{proof}
In this proof, $k$ and $k'$ are exclusively used as integers in the range $1 \le k,k' \le T$.
Since $\chi_{t+T} \ge e^{\nu T} \chi_t$ (see Sec.~\ref{Sec:sim}), the sequence $(c_{k+nT})_n$ is nondecreasing and bounded. Consequently,
\begin{align}
C_k := \lim_{n\rightarrow\infty} c_{k + nT}
\label{Eq:Ck}
\end{align}
is well defined. Note that 
$\max\{C_k : 1 \le k \le T\}$ becomes the optimal upper bound in Lemma~\ref{ideal_conv}.
If $n$ satisfies $ n T>  T'$ with $T' > \max\{t-I_t\}$ (see Remark~\ref{Re:Tp}), then we have
\begin{align}
\nonumber
c_{k+nT} &= e^{-\nu (k+nT)}\max\left \{\frac{k+nT-t'}{\alpha}\chi_{t'} : k+nT-T' \le t' < k+nT\right \}\\
&= \max\left \{\frac{s}{\alpha}e^{-\nu s}c_{k+nT-s} : 1 \le s \le T'\right \}. 
\label{Eq:CT}
\end{align}
Taking $n$ to infinity, we get
%
\begin{align}
C_k = \max \left \{\frac{s}{\alpha}e^{-\nu s}C_{k-s} : 1 \le s \le T'\right \},
\label{Eq:infct}
\end{align}
and, by definition, $C_{k+mT} = C_k$ for any integer $m$.
Since $T e^{-\nu T}/\alpha = 1$ and $C_{k-T} = C_k$, we can rewrite \eqref{Eq:infct} as
\begin{align*}
C_k = \max \left \{\frac{s}{\alpha}e^{-\nu s}C_{k-s} : 1 \le s \le T+1\right \}.
\end{align*}
Comparing terms with $s=T-1, T, T+1$ for any $k$, we have
\begin{align*}
\frac{T-1}{T} e^{\nu} C_{k+1} =\frac{T-1}{\alpha}e^{-\nu (T-1)} C_{k+1}\le C_{k},\\
\frac{T+1}{T} e^{-\nu} C_{k-1} = \frac{T+1}{\alpha}e^{-\nu (T+1)} C_{k-1} \le C_{k},
\end{align*}
which gives
\begin{align*}
\frac{T+1}{T} \le \frac{C_{k+1} e^\nu}{C_k} \le \frac{T}{T-1},
\end{align*}
for all $k$. 
Let $\varphi_k = C_{k} e^\nu / C_{k-1}$ with $\varphi_1 =\varphi_{T+1}= C_1 e^\nu / C_T$. Then,
$$
C_k = C_1 e^{-\nu (k-1)}\prod_{j=2}^{k} \varphi_j,
$$
for $k > 1$. Setting $k=T+1$ and considering $C_{T+1} = C_1$, we have
$$
\prod_{j=1}^T \varphi_j = e^{\nu T} = \frac{T}{\alpha}.
$$

\noindent
To sum up, $C_k$ takes the form
\begin{align*}
C_k = C_0 e^{-\nu k} \prod_{j=1}^k \varphi_j,
\end{align*}
where $C_0$ is a positive constant (note that $\chi_t(\alpha,(C_0 a_n)) = C_0 \chi_t(\alpha,(a_n))$) and $\varphi_k$ satisfies
\begin{align}
\frac{T+1}{T} \le \varphi_k \le \frac{T}{T-1},\quad
\prod_{j=1}^T \varphi_j = \frac{T}{\alpha} = e^{\nu T}.
\label{Eq:restrictphi}
\end{align}
If $\alpha=\alpha_T$ (see Remark~\ref{Rem:alphaT}), then $e^\nu = (T+1)/T$ and the only possible value of 
$\varphi_k$ is $\varphi_k = e^\nu$
for all $k$ because of \eqref{Eq:restrictphi}.\medskip

To simplify \eqref{Eq:CT} for large $n$, we use the following observation.
For $p \in \N$ with $X := 1/T$ and
$C_{k - (T \pm p)} = C_{k \mp p}$, we have
\begin{align*}
\frac{T-p}{\alpha} e^{-\nu(T-p)}C_{k+p}
&= \frac{T-p}{T} e^{\nu p} \frac{C_{k+p}}{C_k} C_k
=C_k \frac{T-p}{T} \prod_{j=1}^p e^{\nu} \frac{C_{k+j}}{C_{k+j-1}}\\
&\le C_k \frac{T-p}{T} \left ( \frac{T}{T-1} \right )^p
= C_k \frac{1-pX}{(1-X)^p},
\end{align*}
and
\begin{align*}
\frac{T+p}{\alpha} e^{-\nu(T+p)}C_{k-p}
&= \frac{T+p}{T} e^{-\nu p} \frac{C_{k-p}}{C_k} C_k
=C_k \frac{T+p}{T} \prod_{j=1}^p \frac{C_{k-j}}{e^\nu C_{k-j+1}}\\
&\le C_k \frac{T+p}{T} \left ( \frac{T}{T+1} \right )^p
= C_k \frac{1+pX}{(1+X)^p}.
\end{align*}
Since \smash{$\sup_{p\ge 2} (1+ pX)/(1+X)^p<1$} for all nonzero $X$ not smaller than $-1$, 
relating $s$ and $p$ by $p=\vert s-T \vert$
we can choose $ \epsilon > 0$ such that  
$C_k - \epsilon > s e^{-\nu s} C_{k-s}/\alpha$ for all $s$  with $\vert s-T\vert  > 1$ and for all $k$.
By \eqref{Eq:Ck}, for this $\epsilon$, there is an integer $m_0$ such that
$C_k - \epsilon < c_{k+nT}  \le C_k$ for all $n \ge m_0$ and for all $k$.
If $n> m_0$, then
$$
c_{k+nT} > C_k - \epsilon > s e^{-\nu s} C_{k-s}/\alpha \ge s e^{-\nu s} c_{k-s+nT}/\alpha,
$$
for all $s$ with $\vert s - T \vert > 1$, which reduces \eqref{Eq:CT} to
\begin{align}
\label{Eq:ckred}
c_{k+(n+1)T} = \max\left \{c_{k+nT}, \tfrac{T+1}{T}e^{-\nu }c_{k-1+nT}, \tfrac{T-1}{T}e^{\nu }c_{k+1+nT} \right \}.
\end{align}
In the following, $n$ is assumed so large that \eqref{Eq:ckred} is valid for all $k$.
Defining \smash{$\delta_{k,n}:=$}$1-c_{k+nT}/C_k$ with the convention $\delta_{k + m T, n} := \delta_{k,n+m}$ for integer $m$ 
and using the definition of $\varphi_k$, we can write
\begin{align*}
\delta_{k,n+1} = \min\left \{ \delta_{k,n} , 1- \tfrac{T+1}{T\varphi_k} (1 - \delta_{k-1,n}),
1 - \tfrac{T-1}{T}\varphi_{k+1} (1 - \delta_{k+1,n} )\right \}.
\end{align*}
As $c_{k+nT} \rightarrow C_k$, we have $\delta_{k,n} \rightarrow 0$ as $n \rightarrow \infty$. Accordingly,
if $(T+1)/(T\varphi_k)<1$, then the term with $(T+1)/(T\varphi_k)$ cannot be a minimum for large $n$.
The same argument is applicable to the term with $(T-1)\varphi_{k+1}/T$.
\medskip


If $\alpha = \alpha_T$, then $\varphi_k = (T+1)/T$ for all $k$ and, accordingly, we have
$\delta_{k,n+1} = \min\{\delta_{k,n},\delta_{k-1,n}\}$ for all $k$ and for all sufficiently large $n$.
If there is $m$ and $k'$ such that $\delta_{k',m}=0$, then $\delta_{k',n}=0$ for all $n \ge m$ and
$\delta_{k'+1,m+1} = 0$, which again gives $\delta_{k'+2,m+2}=0$ and so on. Therefore, we have $\delta_{k,n} =0$ for all $n > m+T$ and 
all~$k$. Hence, to complete the proof for this case, we need to elicit a contradiction if 
$\delta_{k,n}$ is strictly positive for all $n$ and for all $k$.
Since $\delta_{k,n}$ is a nonincreasing sequence of $n$, we have
\begin{align*}
\delta_{k,n+s} &= \min \{ \delta_{k,n+s-1}, \delta_{k-1,n+s-1}\}
= \min\{\delta_{k,n+s-2}, \delta_{k-1,n+s-2}, \delta_{k-1,n+s-1}\}\\
&= \min\{\delta_{k,n+s-2}, \delta_{k-1,n+s-1}\}  
= \min\{\delta_{k,n}, \delta_{k-1,n+s-1}\},
\end{align*}
for all $s \in \N$. Since $\delta_{k-1,n+s-1}$ should approach zero monotonically as $s \rightarrow \infty$,
there should be $s_0$ such that $\delta_{k-1,n+s-1} < \delta_{k,n}$ for all $k$ and for all $s >  s_0$.
Therefore, we get $\delta_{k,n + s+T} = \delta_{k-1,n+s+T-1} = \delta_{k-2,n+s+T-2} = \delta_{k,n+s}$ for all $s>s_0$.
Since $\delta$ cannot increase, we conclude that $\delta_{k,n}$ is a constant
for all sufficiently large $n$. If $\delta_{k,n}$ is strictly positive for all $n$ as assumed, $C_k$ cannot be a limit and
we arrive at a contradiction.
Therefore, there is $t_1$ such that $c_{t+T} = c_t$ for all $t \ge t_1$ in this case.\medskip

If $\alpha\neq \alpha_T$, then there is at least one $\varphi_k$ such that
$(T+1)/T  < \varphi_k$. If $\epsilon$ also satisfies $\epsilon/C_k < 1 - (T+1)/ (T \varphi_k)$, then 
we can write
\begin{align*}
\delta_{k,n+1} = \min\left \{ \delta_{k,n} ,
1 - \tfrac{T-1}{T}\varphi_{k+1} (1 - \delta_{k+1,n} ) \right \},
\end{align*}
for all $n > m_0$.
If $\varphi_{k+1} < T/(T-1)$, then $\delta_{k,n}$ will eventually be smaller than $1-(T-1)\varphi_{k+1}/T$ and 
we have $\delta_{k,n+1} = \delta_{k,n}=0$ for all large $n$.
On the other hand, if $\varphi_{k+1} = T/(T-1) > (T+1)/T$, we have
\begin{align}
\nonumber
\delta_{k,n+1} &= \min \left \{ \delta_{k,n}, \delta_{k+1,n} \right \},\\
\delta_{k+1,n+1} &= \min \left \{ \delta_{k+1,n}, 1 - \tfrac{T-1}{T}\varphi_{k+2} (1 - \delta_{k+2,n} )
\right \}.
\label{Eq:delrec}
\end{align}
Since it is impossible for all $\varphi_k$ to be $T/(T-1)$, there exists $k'$ such that 
$\varphi_{k+i} = T/(T-1)$ for $1 \le i \le k'$ and
$\varphi_{k+k'+1} < T/(T-1)$. Therefore,
$$
\delta_{k+k',n+1} = \delta_{k+k',n},
$$
for all sufficiently large $n$. Once $\delta_{k+k',m} = 0$, then $\delta_{k+i,n} = 0$  for all $0 \le i \le k'$ and for all $n > m+T$ by \eqref{Eq:delrec}. If $\varphi_{k+k'+1} > (T+1)/T$, we can repeat the above procedure.
If $\varphi_{k+k'+1} = (T+1)/T$, we have
$$
\delta_{k+k'+1,n+1} = \min\left \{ \delta_{k+k'+1,n} , \delta_{k+k',n}, 1 - \tfrac{T-1}{T}\varphi_{k+k'+2} 
( 1 - \delta_{k+k'+2,n}) \right \} = 0$$ for all sufficiently large $n$.
Hence, the proof is complete.
\end{proof}

\subsection{Non-uniqueness of periodic solutions}

Proposition~\ref{Prop:periodic} and its proof have shown the general periodic solutions for $t > t_1$ to be of the form
\begin{align*}
\chi_t = \chi_{t_1} \prod_{k=t_1+1}^t \varphi_k,\quad
c_t = c_{t_1} e^{-\nu (t-t_1)} \prod_{k=t_1+1}^t \varphi_k,
\end{align*}
where the $\varphi_k$ satisfy $\varphi_{T+k} = \varphi_k$ and \eqref{Eq:restrictphi}.
Since $(T+1)/T \le e^\nu <T/(T-1)$, setting $\varphi_i = e^\nu$ for all $i$ satisfies \eqref{Eq:restrictphi}, 
which gives the constant sequence $c_{t} = c_{t_1}$. We refer to this solution as the homogeneous state. 
Recall that the homogeneous state is the unique possibility for $\alpha = \alpha_T$, as shown right after \eqref{Eq:restrictphi}.
By constructing an appropriate sequence $(a_n)$, we now show that any set $\{\varphi_k\}$ that satisfies the conditions \eqref{Eq:restrictphi} can give rise to a periodic solution $c_t$. 
Therefore, the periodic solution $c_t$ is not unique and can vary substantially with $(a_n)$ unless $\alpha = \alpha_T$ or $T=1$.
\begin{proposition}
Let
\begin{align}
a_t = \max\left \{\frac{T-i+t-1}{\alpha}
\psi_i \colon 0\le i < T \right \},\quad
\psi_i := \prod_{j=1}^{i} \varphi_j,
\label{Eq:ctex}
\end{align}
with the convention $\prod_{j=1}^0 = 1$.
Then
\begin{align}
\chi_t = \psi_{t+T-1}:= \prod_{j=1}^{t+T-1} \varphi_j,
\label{Eq:indu}
\end{align}
where the $\varphi_j$ are as in \eqref{Eq:restrictphi} with periodicity  $\varphi_{T+j} = \varphi_j$.
\end{proposition}
\begin{proof}
To find $a_1$, we observe that
for $0 \le i < T-1$
\begin{align*}
\frac{T-i-1}{\alpha} \psi_{i+1} &=
\frac{T-i}{\alpha} \psi_i \frac{T-i-1}{T-i} \varphi_{i+1}\\
&\le \frac{T-i}{\alpha} \psi_i \frac{(T-i-1)T}{(T-i)(T-1)}
=\frac{T-i}{\alpha} \psi_i \frac{T^2-(i+1)T}{T^2-(i+1)+i}
\le \frac{T-i}{\alpha}\psi_i,
\end{align*}
where we have used $\varphi_i \le T/(T-1)$.
Therefore, we get
$$
\chi_1 =a_1 =  \frac{T}{\alpha} = \prod_{j=1}^T \varphi_j
=\psi_T,
$$
which is \eqref{Eq:indu} for $t=1$. Note that this $\chi_1$ is trivially valid for $T=1$.
\medskip

Now assume \eqref{Eq:indu} is valid up to $t = n$.
Then,
$$
\chi_{n+1}
= \max\left \{ \frac{T-i+n}{\alpha}\psi_i : 0 \le i \le n+T-1\right \}.
$$
For $i \le n$, we have
\begin{align*}
\frac{T-i+n}{\alpha}\psi_i &= \frac{T-i+n+1}{\alpha}\psi_{i-1}
\frac{T-i+n}{T-i+n+1} \varphi_i \\
&\ge
\frac{T-i+n+1}{\alpha}\psi_{i-1} \frac{(T-i+n)(T+1)}{(T-i+n+1)T}\\
&=\frac{T-i+n+1}{\alpha}\psi_{i-1} \frac{T^2 + (n+1-i)T + n-i}{T^2 + (n+1-i)T }
\\
&\ge \frac{T-i+n+1}{\alpha}\psi_{i-1},
\end{align*}
and
for $ n+T-1>i \ge n$ and $T>1$, we have
\begin{align*}
\frac{T-i+n}{\alpha}\psi_i &= \frac{T-i+n-1}{\alpha}\psi_{i+1}
\frac{T-i+n}{T-i+n-1} \frac{1}{\varphi_{i+1}} \\
&\ge
\frac{T-i+n-1}{\alpha}\psi_{i+1} \frac{(T-i+n)(T-1)}{(T-i+n-1)T}\\
&=\frac{T-i+n-1}{\alpha}\psi_{i+1} \frac{T^2 + (n-1-i)T + i-n}{T^2 + (n-1-i)T }
\\
&\ge \frac{T-i+n-1}{\alpha}\psi_{i+1}.
\end{align*}
Therefore, we have
$$
\chi_{n+1} = \frac{T}{\alpha} \psi_n = \psi_n \prod_{j=n+1}^{T+n} \varphi_j =  \prod_{j=1}^{T+n} \varphi_j.
$$
Induction completes the proof.
\end{proof}

For a realization of \eqref{Eq:ctex} in the branching process, consider an initial condition such that there are $T$ different mutant classes with fitness $f_i:= f^{\psi_i/\alpha}$ ($0\le i < T$)
and the number $N_i$ of individuals with fitness $f_i$ is 
$$
N_i = \left \lfloor f^{(T-i-1)\psi_i/\alpha} \right \rfloor.
$$
Notice that this initial condition with \smash{$\varphi_j = e^{\nu'}$} together with a shift in time
was used in the proof of Lemma~\ref{Th:mainup}. 
In the limit $f\rightarrow \infty$  as in Sec.~\ref{Sec:sim}, $X(t)$ is well approximated by $f^{\chi_t}$ with $a_t$ in \eqref{Eq:ctex}. 
\medskip

In the original branching process, the sequence $(a_n)$ depends both on the initial condition and the stochastic evolution in the early time regime before the deterministic approximation through the recursion relation (\ref{eq:recursion}) becomes valid.
To see this, we recall from Sec.~\ref{Sec:sim} how the recursion relation arises from the stochastic process. 
Since on survival the total population size as well as the largest fitness increases indefinitely,
there should be a generation $t_0$ such that $X(t_0+1) > K$ for any preassigned $K$.
Let $W_0$ be the largest fitness at generation $t_0$,  
define $Y = X(t_0+1)$ and introduce a shifted time variable $t' = t - t_0$.
If $K$ is extremely large, $X(t')$ can be well approximated as
\begin{align*}
X(1) = Y \approx Y^{\chi_1},\quad W_1 \approx Y^{\chi_1/\alpha},\\
X(2) = Y^{a_2} + Y^{\chi_1/\alpha}+1 \approx Y^{\chi_2},\quad W_2 \approx Y^{\chi_2/\alpha},\\
X(3) = Y^{a_3} + Y^{2\chi_1/\alpha}+Y^{\chi_2/\alpha}+1 \approx Y^{\chi_3},\quad W_3 \approx Y^{\chi_3/\alpha}
\end{align*}
where $Y^{a_n}$ is the population size of all mutant classes that appeared prior to generation $t_0$.  Since $Y^{a_n} \le Y W_0^n$, we naturally have $\lim_{n\rightarrow\infty} a_n e^{-\nu n} = 0$, and $(a_n)$ is a permissible sequence that can be entered
into the recursion relation~(\ref{eq:recursion}).

\subsection{Empirical frequency distribution for large $\alpha$}

Whereas the preceding subsection has shown that the empirical frequency distribution
at long times is generally non-universal, we will now argue that it nevertheless has a  well-defined limit for $\alpha \rightarrow \infty$.
Let us begin with the homogeneous state. In this case,
\begin{align*}
\nonumber
J_i(t)= e^{-\nu(t-i)},\quad  P(t) \equiv \alpha(e^\nu -1), \\
R_i(t)= \frac{(t-i)}{\alpha}J_i(t) -1 
= -\frac{1}{\nu \alpha} J_i(t) \log J_i(t) - 1.
\end{align*}
Since $\nu\alpha \rightarrow 1/e$ as $\alpha \rightarrow \infty$,
the homogeneous state for all sufficiently large $\alpha$ is well described by
\begin{align}
R_i \approx -e J_i \log J_i  - 1,
\label{Eq:Rpre}
\end{align}
and the mean log fitness converges to $P = \frac{1}{e}$.\medskip

Moreover, since 
$$ \frac{T}{T-1} - \frac{T+1}{T} = \frac{1}{(T-1) T} = O(T^{-2})$$
and $T/\alpha \rightarrow e$ 
as $\alpha \rightarrow\infty$, in this limit all periodic solutions that satisfy
the constraints \eqref{Eq:restrictphi} become close to homogeneous, 
$\varphi_i = e^{\nu} +O(\alpha^{-2})$.
Therefore, we conjecture that the empirical distribution on survival has \eqref{Eq:Rpre} as a limit 
distribution for $\alpha \rightarrow \infty$.
As an illustration, in Fig.~\ref{Fig:infa} we compare ~\eqref{Eq:Rpre} to numerical solutions of the recursion relation for $\alpha=3$, 4, 5, 6. The numerical data are hardly
distinguishable from \eqref{Eq:Rpre} already for $\alpha = 5$.

\begin{figure}
\centering
\includegraphics[width=\linewidth]{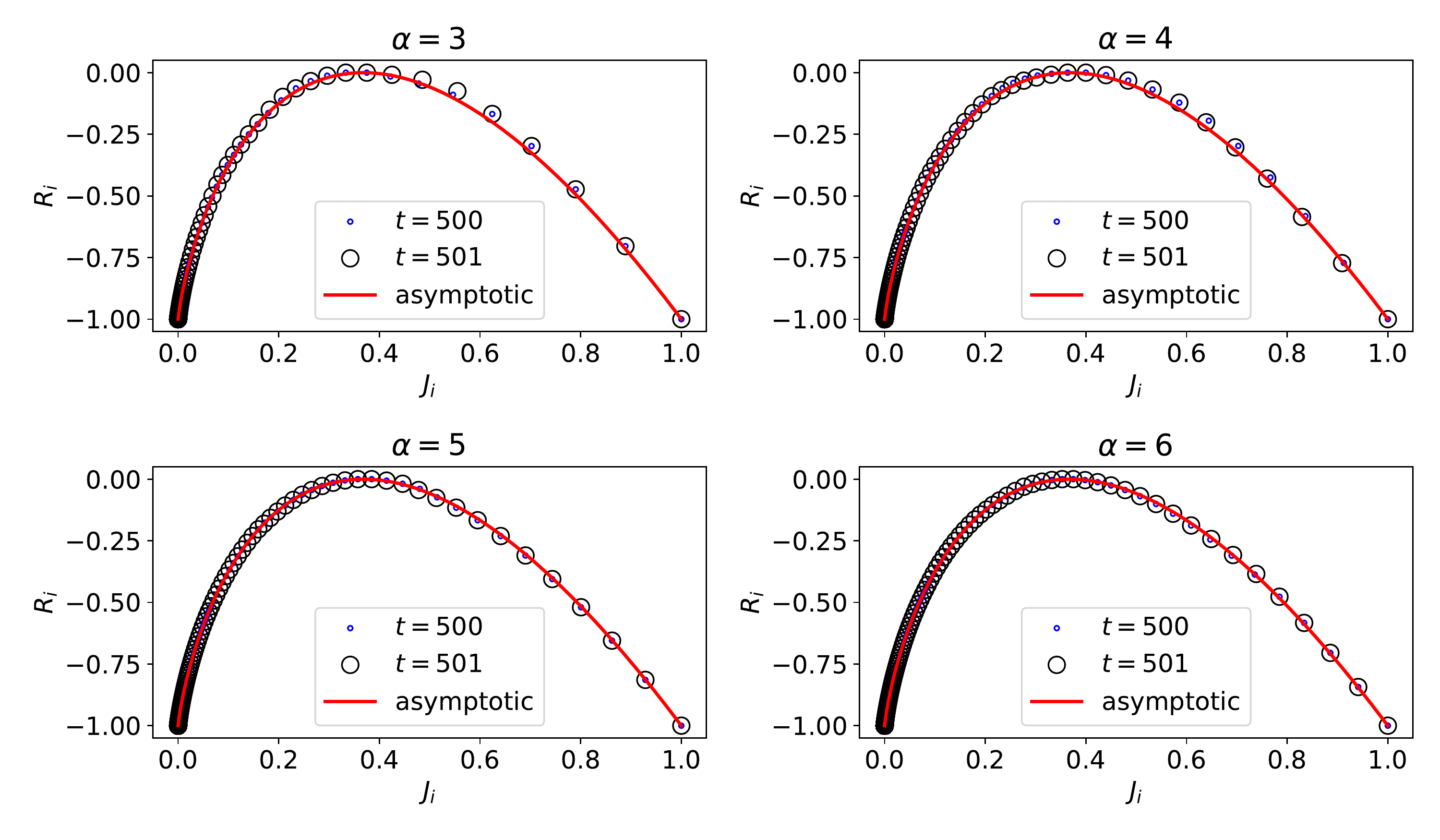}
\caption{\label{Fig:infa} Plot of $R_i$ vs $J_i$  at $t=500$ and 501
for $\alpha=3$, 4, 5, 6. 
For comparison, the asymptotic prediction \eqref{Eq:Rpre} is depicted as a red solid curve. 
For $\alpha=3$ and 4, the changes of the empirical distributions between generations 500 and 501 are still visible, but 
the distributions become indistinguishable from the asymptotic form \eqref{Eq:Rpre} for $\alpha \ge 5$.
}
\end{figure}


\section{\label{Sec:sum}Summary and Discussion}

In this article we have provided a detailed characterization of the superexponential population growth in two closely related stochastic models of evolution. To the best of our knowledge, this is the first rigorous analysis of a branching process with selection and random mutations drawn from an unbounded fitness distribution. A remarkable feature of the models considered here is the emergence of an integer-valued time scale $T$ which depends (discontinuously) on the index $\alpha$ of the underlying Fr\'echet distribution. A partial understanding of the resulting periodic behaviour of the population structure was achieved in a deterministic approximation. Further work on this problem is needed, addressing in particular how the stochastic initial phase of the process determines the non-universal aspects of the asymptotic population distribution.  
\medskip

It is instructive to compare our findings for the branching process to the earlier analysis of a stochastic fixed finite population version of Kingman's model in \cite{Park2008}. In both cases the long-time behavior is dominated by, and can quantitatively understood in terms of extremal mutation events in the past. However, in the fixed finite population model the likelihood of generating mutants that exceed the current population fitness declines with time, and the dynamics reduces to a modified record process, where the takeover of the population by a fit mutant is instantaneous compared to the waiting time for the next fitter mutant. As a consequence, the population at time $t$ is dominated by a mutant that arose at a time of order $t$ in the past. By contrast, in the branching process with Fr\'echet-type distributions, the declining probability of exceeding the current fitness 
is compensated by the rapid growth of the population in such a way that the time lag since the birth of the currently dominant mutant takes on a fixed value $T$.  \medskip

It is reasonable to expect that the growth of the population fitness in the branching process is intermediate between that of the fixed finite population model and the deterministic infinite population model. For Fr\'echet type fitness distributions the deterministic model is ill-defined, but the analysis of the fixed finite population model predicts a polynomial increase of 
the fitness with exponent $1/\alpha$ \cite{Park2008}, which is indeed much slower than the superexponential growth in the branching process. For unbounded Gumbel type distributions the growth law of the fitness is known for infinite as well as for finite populations \cite{Kingman1978,Park2008}. The corresponding behaviour of the branching process will be addressed in future work. 
\bigskip

\bmhead{Acknowledgments}

S-CP acknowledges the support by the National Research Foundation of Korea (NRF) grant funded by the Korea government (MSIT) (Grant No. 2020R1F1A1077065) and by the Catholic University of Korea, research fund 2021.
JK and PM were supported by the German Excellence Initiative through the UoC Forum
``Classical and Quantum Dynamics of Interacting Particle Systems''.

\bibliography{f}

\end{document}